\documentclass[12pt, reqno]{amsart}
\usepackage{amsmath, amsthm, amscd, amsfonts, amssymb, graphicx, color}
\usepackage[bookmarksnumbered, colorlinks, plainpages]{hyperref}
\hypersetup{colorlinks=true,linkcolor=red, anchorcolor=green, citecolor=cyan, urlcolor=red, filecolor=magenta, pdftoolbar=true}
\input{mathrsfs.sty}
\usepackage{xcolor}
\usepackage{tikz}
\usetikzlibrary{decorations.pathreplacing,angles,quotes}
\definecolor{cof}{RGB}{219,144,71}
\definecolor{pur}{RGB}{186,146,162}
\definecolor{greeo}{RGB}{91,173,69}
\definecolor{greet}{RGB}{52,111,72}
\usepackage{xcolor}
\usepackage{ulem}
\newtheorem{theorem}{Theorem}[section]
\newtheorem{lemma}[theorem]{Lemma}
\newtheorem{proposition}[theorem]{Proposition}
\newtheorem{cor}[theorem]{Corollary}
\theoremstyle{definition}
\newtheorem{definition}[theorem]{Definition}

\theoremstyle{remark}
\newtheorem{remark}[theorem]{\bf{Remark}}
\numberwithin{equation}{section}
\newcommand{\sm}{{\rm Sm}\,}
\newcommand{\ksm}{{\rm k\text{-}Sm}\,}
\newcommand{\rint}{{\rm Int_r}\,}
\newcommand{\ext}{{\rm Ext}\,}
\newcommand{\dm}{{\rm dim}\,}
\newcommand{\aff}{\text{aff}}
\newcommand{\co}{ \text{co}}
\newcommand{\spn}{{\rm span}}
\allowdisplaybreaks
\begin{document}
	
	\title[Refinements of the Blanco-Koldobsky-Turn\v{s}ek Theorem]{Refinements of  the Blanco-Koldobsky-Turn\v{s}ek Theorem }
	\author[ Manna, Mandal, Paul and Sain  ]{Jayanta Manna, Kalidas Mandal,    Kallol Paul and Debmalya Sain }

    \address[Mandal]{Department of Mathematics\\ Jadavpur University\\ Kolkata 700032\\ West Bengal\\ INDIA}
	\email{kalidas.mandal14@gmail.com}
	\address[Manna]{Department of Mathematics\\ Jadavpur University\\ Kolkata 700032\\ West Bengal\\ INDIA}
	\email{iamjayantamanna1@gmail.com}

	\address[Paul]{Vice-Chancellor, Kalyani University \& Professor of Mathematics (on lien) \\ Jadavpur University \\ Kolkata \\ West Bengal\\ INDIA}
	\email{kalloldada@gmail.com}
	
		\newcommand{\acr}{\newline\indent}
	\address[Sain]{Department of Mathematics\\ Indian Institute of Information Technology, Raichur\\ Karnataka 584135 \\INDIA}
	\email{saindebmalya@gmail.com}

	\subjclass[2010]{Primary 46B20,  Secondary 46B04}
	\keywords{Isometry; preservation of orthogonality; norm derivatives; polyhedral Banach spaces }	
	
	\begin{abstract}
		We refine  the well-known Blanco-Koldobsky-Turn\v{s}ek theorem which states that a norm one linear operator defined on a Banach space is an isometry if and only if it preserves orthogonality at every element of the space.  We improve the result for  Banach spaces in which the set of  all smooth points forms a dense $G_{\delta}$ set. We further demonstrate that if a norm one operator preserves orthogonality on a hyperplane not passing through the origin then it is an isometry. In the context of finite-dimensional Banach spaces, we prove that preserving orthogonality on the set of all extreme points of the unit ball forces the norm one operator to be an isometry, which substantially refines the Blanco-Koldobsky-Turn\v{s}ek theorem. Finally, for finite-dimensional polyhedral spaces, we establish the significance of the set of all $k$-smooth points for any possible $k,$ in the study of isometric theory.
	\end{abstract}

	\maketitle
	
	\section{Introduction.}
	
	The famous Blanco-Koldobsky-Turn\v{s}ek characterization \cite{BT06,K93} of isometries as norm one linear operators that preserve Birkhoff-James orthogonality at each element of the space, is one of the major achievements in the isometric theory of Banach spaces. Due to appropriate importance of this  result on global orthogonality preservation, it is worth exploring local versions of this phenomenon. Recently in \cite{SMP24}, some  refinements of the theorem was obtained in the setting of finite-dimensional Hilbert spaces and some special finite-dimensional polyhedral spaces. The purpose of this article is to substantially refine the Blanco-Koldobsky-Turn\v{s}ek characterization theorem for general Banach spaces. We now introduce the notations and terminologies to the readers which are essential to present this article. 
	
	The letters $ \mathbb{X}, \mathbb{Y} $ denote real Banach spaces and $\mathbb{X}^*$ denotes the dual space of $\mathbb{X}.$ Let $B_{\mathbb{X}}$ and $S_{\mathbb{X}}$ be the unit ball and the unit sphere of $\mathbb{X}$, respectively.  The convex hull of a nonempty set $D\subset \mathbb{X}$ is denoted by $\co(D).$
	 Let us recall that an affine space is defined as a translation of a subspace of the vector space. The intersection of all affine spaces containing a nonempty set $D \subset \mathbb{X}$ is denoted by $\aff(D).$ By $\text{Int}_r~D,$  we denote the relative interior of a nonempty set $D \subset \mathbb{X},$ i.e., $\text{Int}_r~D=\{x\in D : \text{ there exists }\epsilon>0\text{ such that } B(x,\epsilon)\cap \aff(D)\subseteq D\}.$ 
     For a nonempty convex set $D,$ let $\ext D$ be the collection of all extreme points of $D.$ 
     If for a finite-dimensional Banach space $\mathbb{X},$ $\ext B_{\mathbb{X}}$ is finite then $\mathbb{X}$ is considered a polyhedral Banach space.  A convex subset $F$ of a convex set $G\subset \mathbb{X}$ is said to be a face of $G$ if for any $u,v\in G,~ (1-t)u+v\in F$ implies that $u,v\in F,$ where $0<t<1.$   The dimension of a face $F$ is defined as the dimension of the subspace $\mathbb{V}=\spn \{u-v : u, v \in F\}.$  If the dimension of the face $F$ is $k,$ then $F$ is called a $k$-face of $G.$ A maximal face of $G$ is called a facet of $G$.
	
     For any non-zero $ z\in \mathbb{X},$ we define  $J(z)$  as the collection of all supporting functionals at $z,$ i.e., $ J(z) = \{ f \in S_{\mathbb{X}^*} : f(z) = \| z \| \}. $  A non-zero  $ z \in \mathbb{X} $ is  $ k $-smooth point (or order of smoothness is $k$)  if $\dm\spn \,J(z)=k,$ where $k$ is a natural number. In particular, $1$-smooth points are called smooth points.  In case of an $n$-dimensional polyhedral Banach space $\mathbb{X},$ a point $z\in S_{\mathbb{X}}$ is $k$-smooth if and only if $z$ is in the relative interior of an $(n-k)$-face of $B_{\mathbb{X}}$ \cite{SSP24}.  For more insights on smooth and $k$-smooth points in Banach spaces and their various applications, the readers can see the articles \cite{DMP22,KS05,LR07,MPD22,MP20,PSG16,SPMR20,W18}.  We denote the collection of all smooth points in $\mathbb{X}$ by $\sm\mathbb{X}$ and the collection of all $k$-smooth points in $\mathbb{X}$ by $\ksm\mathbb{X}.$
	
	For a real normed linear space $\mathbb {X},$  the two mappings $\rho'_+, \rho'_-: \mathbb{X}\times \mathbb{X}\rightarrow \mathbb{R}$ are defined by
	\[\rho'_+(u,v)=\|u\|\lim_{t\to 0^+}\frac{\|u+tv\|-\|u\|}{t}\text{ and } \rho'_-(u,v)=\|u\|\lim_{t\to 0^-}\frac{\|u+tv\|-\|u\|}{t},\] where $u,v\in\mathbb{X}.$ These mappings are known as norm derivatives at $u$ in the direction of $v.$
	The mapping $\rho': \mathbb{X}\times \mathbb{X}\rightarrow \mathbb{R}$ defined by
	\[\rho'(u,v)=\frac{1}{2}\rho'_+(u,v) + \frac{1}{2}\rho'_-(u,v)\] is called $M$-\textit{semi inner product}. 
	The mappings $\rho'_{\pm}, \rho'$ are continuous with respect to the second variable, although they may not exhibit continuity in the first variable. The norm derivatives are widely used in exploring the geometry of normed linear spaces, see \cite{A86, AST09, CW13, M87, SA21}. 
	
	Given $u,v\in\mathbb{X}, \rho_+$-orthogonality, $\rho_-$-orthogonality and $\rho$-orthogonality are defined as	\[u\perp_{\rho_+}v\Longleftrightarrow \rho'_+(u,v)=0,~u\perp_{\rho_-}v\Longleftrightarrow \rho'_-(u,v)=0\text{ and }u\perp_{\rho}v\Longleftrightarrow \rho'(u,v)=0,\] respectively.

	Given  $u, v \in \mathbb{X},$  $u$ is Birkhoff–James orthogonal \cite{B35, J47}  to $v,$ written as $ u \perp_B v,$ if $ \| u + \lambda v\| \geq \|u\|$ for all scalars $\lambda. $ According to the  Hahn-Banach theorem,  for a given $u \in \mathbb{X},$ there exist enough $v \in \mathbb{X}$ such that $ u \perp_B v.$ Let $ u^{\perp_B} $ be the collection of all  $v\in \mathbb{X}$ such that $u \perp_B v.$ It is straightforward to observe that the Birkhoff-James orthogonality relation is homogeneous, i.e., $u\perp_B v\implies \alpha u\perp_B \beta v,$ for all scalars $\alpha, \beta.$ In addition, in the case of an inner product space, the Birkhoff-James orthogonality  coincides with the usual inner product orthogonality.  For more on Birkhoff-James orthogonality,  readers can consult the recently published book \cite{Book24}.
	
	The letter $T$ is reserved for a bounded linear operator either from $\mathbb{X}$ to $\mathbb{Y}$ or from $\mathbb{X}$ to $\mathbb{X},$ which will be clear from the context. Consequently we write $  T \in \mathbb{L}(\mathbb{X}, \mathbb{Y}) $ or $  T \in \mathbb{L}(\mathbb{X}).$ We say that $ T $ preserves $\rho_{\pm}$ -orthogonality and $\rho$-orthogonality if $ u, v \in \mathbb{X},$ $ u \perp_{\rho_\pm} v \implies Tu   \perp_{\rho_\pm} Tv$ and $ u \perp_{\rho} v \implies Tu   \perp_{\rho} Tv,$ respectively. Similarly, we say that $T$ preserves Birkhoff-James orthogonality if for $ u, v \in \mathbb{X},$ $ u \perp_B v \Rightarrow Tu \perp_B Tv. $ For the local version of this property, $T$ is said to preserve Birkhoff-James orthogonality at $ u \in \mathbb{X} $ if for all $v \in \mathbb{X},$ $ u \perp_B v \Rightarrow Tu \perp_B Tv. $ Several recent studies \cite{MMPS25,S20,S18,SMP24,SRT21} have demonstrated the importance of the local preservation of Birkhoff-James orthogonality in understanding the geometric properties of the underlying Banach spaces. We recall the following definition introduced in \cite{SMP24}, which is connected with the main objectives of this article: 
	\begin{definition}\label{k-set}
		A set $A\subset S_{\mathbb{X}}$ is called a $\mathcal{K}$-set (Koldobsky-set) if for any operator $  T \in \mathbb{L}(\mathbb{X}) ,$  preservation of Birkhoff-James orthogonality at each point of $A$  under $T$ implies that $T$ is an isometry up to scalar multiplication.
	\end{definition}
	It is evident that identifying a $\mathcal{K}$-set $A$ for a Banach space $\mathbb{X}$, where $A \subsetneq S_{\mathbb{X}}$, will undoubtedly refine the Blanco-Koldobsky-Turnšek theorem for the particular Banach space $\mathbb{X}$. The primary objective of this article is to explore the $\mathcal{K}$-sets in Banach spaces in which the set of all smooth points forms a dense $G_{\delta}$ set. It is clear that our approach is related to the local preservation of Birkhoff-James orthogonality. In this article, we first show that in Banach spaces in which the set of all smooth points forms a dense $G_{\delta}$ set, the set of all smooth points of the unit sphere constitutes a $\mathcal{K}$-set. Furthermore, we establish that any dense subset of the unit sphere is a $\mathcal{K}$-set in such spaces. Next, we establish that for a Banach space whose unit ball is the convex hull of its extreme points, the set of all extreme points of the unit ball also forms a $\mathcal{K}$-set. Finally, we show that in a finite-dimensional polyhedral Banach space, the set of all $k$-smooth points is a $\mathcal{K}$-set, for any possible $k.$ 
	
	\section{Main Results}
    We start this section with the following relation between preservation of Birkhoff-James orthogonality at a point by a bounded linear operator and the smoothness of image of that point. 
\begin{proposition}\label{inverse image}
    Let $T \in \mathbb{L}(\mathbb{X}, \mathbb{Y})$ be such that $T$ preserves Birkhoff-James orthogonality at $x\in \mathbb{X}.$  Then $x$ is smooth if $Tx$ is smooth.
\end{proposition}
\begin{proof}
 Let $Tx$ be smooth. Then there exists $\psi\in S_{\mathbb{Y}^*}$ such that $J(Tx)=\{\psi\}.$ If possible suppose that $x$ is not smooth. Then there exist two linearly independent $f$ and $g$ in $J(x)$. Therefore, there exists $y \in \mathbb{X}$ such that $y \in \ker f \setminus \ker g$. Then $\mathbb{X}=\spn\,\{y\}\oplus\ker g$ and so there exist $\alpha(\neq0)\in \mathbb{R}$ and $h(\neq 0)\in \ker g$ such that $x=\alpha y+h.$  Since $T$ preserves Birkhoff-James orthogonality at $x$, it follows that $T(x^{\perp_B})\subset (Tx)^{\perp_B}$ and so $T(\ker f \cup \ker g)\subset \ker \psi.$ This implies that $\psi(Tx)=\psi(\alpha Ty+Th)=0.$ This contradicts the fact that $J(Tx)=\{\psi\}.$ Therefore, $x$ is smooth.
\end{proof}

	Now we aim to demonstrate that in a  Banach space, any operator that preserves Birkhoff-James orthogonality on a given set necessarily preserves this orthogonality at every smooth point within the closure of that set. 
	\begin{theorem}
    \label{sm pre}
		Let $A$ be a subset of $\mathbb{X}.$  If $T\in \mathbb{L}(\mathbb{X},\mathbb{Y})$ preserves Birkhoff-James orthogonality at each point of $A,$ then it preserves the same at each point of $\overline{A}\cap \sm\mathbb{X}.$
	\end{theorem}
	\begin{proof}
	     Let  $T\in \mathbb{L}(\mathbb{X},\mathbb{Y})$ preserve Birkhoff-James orthogonality at each point of $A.$ If $\overline{A}\cap \sm\mathbb{X}=\emptyset$ then we are done. Let $\overline{A}\cap \sm\mathbb{X}\neq\emptyset.$ Let $x\in \overline{A}\cap \sm\mathbb{X}.$ Then  there exists a sequence $\{x_n\}\subset A$ such that $x_n\longrightarrow x.$ Let for each $n\in \mathbb{N},$ $f_n\in J(x_n).$ Since $B_{\mathbb{X}^*}$ is weak* compact, there exists a weak* convergent subnet $\{f_{\alpha}\}_{\alpha \in \Lambda}$ of $\{f_n\}.$  Suppose that $f_{\alpha}\overset{w*}{\longrightarrow} f\in B_{\mathbb{X}^*}.$ Consider the subnet $\{x_{\alpha}\}$ of $\{x_n\},$  then $x_{\alpha}\longrightarrow x.$  Thus, 
		$f_{\alpha}(x_{\alpha})\longrightarrow f(x).$ 
		Since $f_{\alpha}(x_{\alpha})=\|x_{\alpha}\|$ for each $\alpha\in \Lambda,$ it follows that $f(x)=\|x\|.$ This implies that $J(x)=\{f\}.$ Now we show that $T$ preserves Birkhoff-James orthogonality at $x.$  Let $x\perp_{B}z$ and so $z\in \ker f.$ We claim that there exist $z_{\alpha}\in \ker f_{\alpha}$ for all $\alpha\in \Lambda$ such that $z_{\alpha}\longrightarrow z.$  Let \[y\in\mathbb{X}\setminus\Big (\big(\bigcup\limits_{\alpha\in \Lambda} \ker f_{\alpha}\big)\bigcup \ker f\Big ).\]
        Then for each $\alpha\in \Lambda,$ there exists $r_{\alpha}\in \mathbb{R}$ such that $f_{\alpha}(z)=r_{\alpha}f_{\alpha}(y).$ Then \[r_{\alpha}f_{\alpha}(y)=f_{\alpha}(z)\longrightarrow f(z)= 0.\] Since $f_{\alpha}(y)\longrightarrow f(y)\neq 0,$ it follows that $r_{\alpha}\longrightarrow 0.$ For each $\alpha\in \Lambda,$ consider $z_{\alpha}=z-r_{\alpha} y.$  Then for each  $\alpha \in \Lambda,$ $z_{\alpha}\in \ker f_{\alpha}$ and $z_{\alpha}\longrightarrow z.$ Thus, our claim is established. 
         Now, $Tx_{\alpha}\perp_{B}Tz_{\alpha}$ for all $\alpha\in \Lambda.$ Then for each $\alpha\in \Lambda,$
		\[ \|Tx_{\alpha}+\lambda Tz_{\alpha}\|\geq \|Tx_{\alpha}\|\text{ for all }\lambda\in \mathbb{R}.\] This implies that
		\[ \|Tx+\lambda Tz\|\geq \|Tx\|\text{ for all }\lambda\in \mathbb{R}.\]
		So $Tx\perp_{B} Tz.$ Thus,  $T$ preserves Birkhoff-James orthogonality at $x$ and this completes the proof.
	\end{proof}

Observe that the Birkhoff-James orthogonality preserving set $P_{B}=\{x\in \mathbb{X}: x\perp_{B}y\implies Tx\perp_{B} Ty\}$ for an operator $T$ may not be closed in $\mathbb{X}.$ For instance, take  $\mathbb{X}=\ell_{\infty}^2$ and define the operator $T(x,y)=(2x,y).$ Then $P_B$ is not closed in $\mathbb{X}.$ However, using Theorem \ref{closed} we can show that the set $P_{B}$  is closed with respect to the subspace topology of $\sm\mathbb{X}.$

    \begin{cor}\label{closed}
        Let $T\in \mathbb{L}(\mathbb{X},\mathbb{Y}).$  Then $P_{B}\cap \sm\mathbb{X}$ is closed in $\sm\mathbb{X}.$
    \end{cor}
    \begin{proof}
       Since $T$ preserves Birkhoff-James orthogonality at each point of $P_{B},$  it follows from Theorem \ref{closed} that $T$ preserves the same at each point of $\overline{P_{B}}\cap \sm\mathbb{X}.$ This implies that $\overline{P_{B}}\cap \sm\mathbb{X}\subset P_{B}.$ So,
       \[P_{B}\cap \sm\mathbb{X}\subset \overline{P_{B}}\cap \sm\mathbb{X}\subset P_{B}\cap \sm\mathbb{X}.\]
       Therefore, $P_{B}\cap \sm\mathbb{X}= \overline{P_{B}}\cap \sm\mathbb{X}$ and this completes the proof.
    \end{proof}

	Now, we focus on studying the $\mathcal{K}$-sets in Banach spaces  in which the set of  all smooth points forms a dense $G_{\delta}$ set. For this, we need the following lemma.  
	\begin{lemma}\label{sm dense}
		Let  $T\in \mathbb{L}(\mathbb X,\mathbb{Y})$ be anon-zero operator. Then the following results hold:
        \begin{itemize}
            \item[(i)] If $\sm\mathbb{X}$ is dense in $\mathbb{X}$ then $\sm\mathbb{X}\setminus\ker T $ is also dense in $\mathbb{X}.$
            \item[(ii)] If $\sm\mathbb{X}$ and $\sm\mathbb{Y}$ are  dense $G_{\delta}$ subsets of $\mathbb{X}$ and $\mathbb{Y},$ respectively and $T$ is bijective then $\sm\mathbb{X}\cap T^{-1}(\sm\mathbb{Y})$ is dense in $\mathbb{X}.$ 
        \end{itemize}
	\end{lemma}
	\begin{proof}
		(i) Let $\sm\mathbb{X}$ be dense in $\mathbb{X}.$ Then
		
		\[\mathbb{X}\setminus \ker T=\overline{\sm\mathbb{X}}\setminus\ker T\subset \overline{\sm\mathbb{X}\setminus\ker T}.\]
		Thus, $\sm\mathbb{X}\setminus\ker T$ is dense in $\mathbb{X}\setminus \ker T.$ Again $\mathbb{X}\setminus \ker T$ is dense in $\mathbb{X}.$ Therefore, $\sm\mathbb{X}\setminus\ker T $ is dense in $\mathbb{X}.$\\

        (ii)  Let $\sm\mathbb{X}$ and $\sm\mathbb{Y}$ be  dense $G_{\delta}$ subsets of $\mathbb{X}$ and $\mathbb{Y},$ respectively. Then  $\mathbb{Y}\setminus\sm\mathbb{Y}$ is of the first category in $\mathbb{Y}.$ Suppose $T$ is bijective.  Then $T$ is a homeomorphism and therefore,  $T^{-1}(\mathbb{Y}\setminus \sm\mathbb{Y})=\mathbb{X}\setminus T^{-1}(\sm\mathbb{Y})$ is of the first category in $\mathbb{X}.$ Again   $\mathbb{X}\setminus\sm\mathbb{X}$ is of the first category in $\mathbb{X}.$ Now, 
		\[\mathbb{X}\setminus (\sm\mathbb{X}\cap T^{-1}(\sm\mathbb{Y}))=(\mathbb{X}\setminus \sm\mathbb{X})\cup \mathbb{X}\setminus T^{-1}(\sm\mathbb{Y}).\]
		Thus, $\mathbb{X}\setminus (\sm\mathbb{X}\cap T^{-1}(\sm\mathbb{Y}))$  is of the first category  in $\mathbb{X}.$ Therefore, $\sm\mathbb{X}\cap T^{-1}(\sm\mathbb{Y})$ is dense in $\mathbb{X}.$
	\end{proof}

	In the following theorem, we provide a sufficient condition for an operator to be injective.
	\begin{theorem}\label{sm inje}
		Let $\mathbb{X}$ be a Banach space such that $\sm\mathbb{X}$ is dense in $\mathbb{X}.$ If a non-zero operator $T\in \mathbb{L}(\mathbb X,\mathbb{Y})$ preserves  Birkhoff-James orthogonality at each $x\in \sm\mathbb{X}$ then $T$ is injective.
	\end{theorem}
	\begin{proof}
		Let $T(\neq 0)\in \mathbb{L}(\mathbb X,\mathbb{Y})$  preserve Birkhoff-James orthogonality at each $x\in \sm\mathbb{X}.$ We show that $T $ is injective. On the contrary suppose that $\ker T\neq \emptyset.$ Suppose that $z(\neq 0)\in \ker T.$ We claim that there exists $x\in \sm\mathbb{X}\setminus \ker T $ such that $x\not \perp_{B} z.$ On the contrary to our claim suppose that  for all $x\in \sm\mathbb{X}\setminus \ker T,$ $x\perp_{B}z.$ From  Lemma \ref{sm dense} (i), it follows that there exists a sequence $\{z_n\}\subset \sm\mathbb{X}\setminus\ker T$ such that $z_n\longrightarrow z.$ Then for all $n\in \mathbb{N},$ $z_n\perp_{B}z$ and so $\|z_n+\lambda z\|\geq \|z_n\| $ for all $\lambda\in \mathbb{R}.$ Taking $n\longrightarrow \infty,$ we have $\|z+\lambda z\|\geq \|z\| $ for all $\lambda\in \mathbb{R}.$ This implies that $z=0,$ a contradiction. Thus, our claim is established. Let $x\in \sm\mathbb{X}\setminus \ker T $  be such that $x\not \perp_{B} z.$ Then there exists $\alpha\in\mathbb{R}\setminus\{0\}$ such that $x\perp_{B}\alpha x +z.$ This implies that $Tx\perp_{B}\alpha Tx,$ which is a contradiction. Therefore, $T$ is injective.
	\end{proof}

    Now, we recall some important properties regarding the norm derivatives. The proofs can be found in \cite{AST09}.
	\begin{proposition}
		Given $u,v\in\mathbb{X}$ and any $\alpha\in \mathbb{R},$ the followings are hold:
		\begin{itemize}
			\item[(i)] $ \rho'_{\pm}(\alpha u, v)=\rho'_{\pm}( u, \alpha v)=\begin{cases}
				\alpha \rho'_{\pm}( u, v), &\text{ if }\alpha\geq 0\\
				\alpha \rho'_{\mp}( u, v), &\text{ if }\alpha< 0.
			\end{cases}$\\ In particular, it follows that $\rho'(\alpha u, v)=\alpha \rho'( u, v)=\rho'\text{ for all } \alpha\in \mathbb{R}.$
			\item[(ii)] $ \rho'_{\pm}( u,\alpha u +v)= \alpha \|u\|^2+ \rho'_{\pm}( u, v).$
			\item[(iii)] $ \rho'_+( u, v)= \|u\|\sup\big\{f(v):f\in J(u)\big\}.$
			\item[(iv)] $ \rho'_-( u, v)= \|u\|\inf\big\{f(v):f\in J(u)\big\}.$
			\item[(v)] $ \rho'_-(u,v)\leq \rho'_+(u,v).$ Moreover, $\mathbb{X}$ is smooth if and only if $ \rho'_-(u,v)= \rho'_+(u,v),$ for all $u,v\in\mathbb{X}.$
			\item[(vi)]  $u\perp_B v\iff \rho'_-(u,v)\leq 0 \leq \rho'_+(u,v).$ Moreover, either of the conditions $u\perp_{\rho_+}v$ or $u\perp_{\rho_-}v$ implies that $u\perp_B v.$
		\end{itemize}
	\end{proposition}
	
	We next find a sufficient condition for a bijective operator between Banach spaces in which the set of  all smooth points forms a dense $G_{\delta}$ set to be a scalar multiple of an isometry.
	
	\begin{theorem}\label{sep iso}
		 Let $\sm\mathbb{X}$ and $\sm\mathbb{Y}$ be dense $G_{\delta}$ subsets of $\mathbb{X}$ and $\mathbb{Y},$ respectively. If a bijective operator $T\in \mathbb{L}(\mathbb{X},\mathbb{Y})$ preserves Birkhoff-James orthogonality at each point of $\sm\mathbb{X}\cap T^{-1}(\sm\mathbb{Y})$ then $T$ is a scalar multiple of an isometry.
	\end{theorem}
	\begin{proof}
    Let $T\in \mathbb{L}(\mathbb{X},\mathbb{Y})$ be bijective and preserve Birkhoff-James orthogonality at each  point of $\sm\mathbb{X}\cap T^{-1}(\sm\mathbb{Y}).$ Let $D=\sm\mathbb{X}\cap T^{-1}(\sm\mathbb{Y}).$    Let $u\in D.$ Then there exists $w(\neq 0)\in \mathbb{X}$  such that $u\perp_{B}w.$ Let $v\in \mathbb{X}\setminus\{0\}.$ Then $v= \alpha u+w $ for $\alpha\in \mathbb{R}.$ Since $u$ is smooth, it follows that  $u\perp_{B}w\implies u\perp_{\rho_+}w\implies \rho'_+(u, w)=0.$  
		Therefore, \[\rho'_+(u, v)= \rho'_+(u, \alpha u +w)=\alpha \|u\|^2+ \rho'_+(u, w)=\alpha \|u\|^2.\]
		As $u\in D$, so $Tu$ is smooth. Now,  \[u\perp_{B} w\implies  Tu\perp_{B}Tw\implies Tu\perp_{\rho_+}Tw\implies \rho'_+(Tu, Tw)=0.\]
		So \[\rho'_+(Tu, Tv)= \rho'_+(Tu, \alpha Tu +Tw)=\alpha \|Tu\|^2+ \rho'_+(Tu, Tw)=\alpha \|Tu\|^2.\]
		Therefore, for all $u\in D,v\in\mathbb{X}\setminus\{0\},$ 
		\[\rho'_+( Tu, Tv)=\frac{\|Tu\|^2}{\|u\|^2} \rho'_+( u, v).\] 
		Let us consider the function $f:\mathbb{X}\setminus\{0\}\longrightarrow \mathbb{R}$ defined by
		$f(x)=\frac{\|Tx\|}{\|x\|}.$ Clearly, $f$ is a continuous real-valued function on $\mathbb{X}\setminus\{0\}.$ Let $ x\in D$ and $y\in \mathbb{X}.$ 	Then, 
		\begin{eqnarray*}
			&&\lim\limits_{t\to 0^+}\frac{f(x+ty)-f(x)}{t}\\
			&&~=~\lim\limits_{t\to 0^+}\frac{1}{t}\Bigg(\frac{\|Tx+tTy\|}{\|x+ty\|}-\frac{\|Tx\|}{\|x\|}\Bigg)\\
			&&~=~\lim\limits_{t\to 0^+}\frac{1}{t}\Bigg(\frac{\|Tx+tTy\|\|x\|-\|Tx\|\|x\|-\|Tx\|\|x+ty\|+\|Tx\|\|x\|}{\|x+ty\|\|x\|}\Bigg)\\
			&&~=~\lim\limits_{t\to 0^+}\frac{1}{t}\Bigg\{\frac{\|x\|\big(\|Tx+tTy\|-\|Tx\|\big)-\|Tx\|\big(\|x+ty\|-\|x\|\big)}{\|x+ty\|\|x\|}\Bigg\}\\
			&&~=~\frac{\|Tx\|}{\|x\|}\Bigg\{\frac{\|Tx\|}{\|Tx\|^2}\lim\limits_{t\to 0^+}\frac{\|Tx+tTy\|-\|Tx\|}{t}-\frac{\|x\|}{\|x\|^2}\lim\limits_{t\to 0^+}\frac{\|x+ty\|-\|x\|}{t}\Bigg\}	\\
			&&~=~\frac{\|Tx\|}{\|x\|}\Bigg\{\frac{1}{\|Tx\|^2}\rho'_+( Tx, Ty)-\frac{1}{\|x\|^2}\rho'_+( x, y)\Bigg\}\\
			&&~=~0.
		\end{eqnarray*}
		By a similar calculation, we also have \[\lim\limits_{t\to 0^-}\frac{f(x+ty)-f(x)}{t}=0.\] 
		This shows that $f$ is constant on $D.$ Since $T$ is bijective, it follows from Lemma \ref{sm dense} (ii) that $D$ is dense in $\mathbb{X}.$ Now as $f$ is a continuous real-valued function on  $\mathbb{X}\setminus\{0\},$ so we conclude that $f$ is constant on $\mathbb{X}\setminus\{0\}.$ Therefore, for all $x\in \mathbb{X}\setminus\{0\},$ $\frac{\|Tx\|}{\|x\|}=f(x)=\lambda$ (say). This implies that $\|Tx\|=\lambda\|x\|.$ 
		
	\end{proof}
	
	As a consequence of the above theorem, we obtain a refinement of the Blanco-Koldobsky-Turn\v{s}ek theorem in  Banach spaces in which the set of  all smooth points forms a dense $G_{\delta}$ set.
	\begin{cor}\label{sm k set}
		 Let $\sm\mathbb{X}$ be a  dense $G_{\delta}$ subset of $\mathbb{X}.$ If $T\in \mathbb{L}(\mathbb{X})$ preserves Birkhoff-James orthogonality at each point of $\sm\mathbb{X}$ then $T$ is a scalar multiple of an isometry, i.e., $\sm\mathbb{X}\cap S_{\mathbb{X}}$ is a $\mathcal{K}$-set.
	\end{cor}
    \begin{proof}
        Let $T\in \mathbb{L}(\mathbb{X})$ preserve Birkhoff-James orthogonality at each point of $\sm\mathbb{X}.$ From Theorem \ref{sm inje}, it follows that $T$ bijective from $\mathbb{X}$ to $T(\mathbb{X}).$ Now, $T$ preserves Birkhoff-James orthogonality at each point of $\sm\mathbb{X}\cap T^{-1}(\sm\mathbb{T(\mathbb{X})}).$  Thus, it follows from Theorem \ref{sep iso} that $T$ is a scalar multiple of an isometry.  
        \end{proof}

	Now we proceed to offer another refinement of the Blanco-Koldobsky-Turn\v{s}ek theorem in those spaces.
	\begin{theorem}\label{Th preseves}
		Let $\sm\mathbb{X}$ be a  dense $G_{\delta}$ subset of $\mathbb{X}.$  Suppose  $U\subset \mathbb{X}$ is dense in $\mathbb{X}.$ If $T\in \mathbb{L}(\mathbb{X})$ preserves Birkhoff-James orthogonality at each point of $U$ then $T$ is a scalar multiple of an isometry. 
	\end{theorem}
	\begin{proof}
		Let $T\in \mathbb{L}(\mathbb{X})$ preserves Birkhoff-James orthogonality at each point of $U.$  Then it follows from Theorem \ref{sm pre} that  $T$ preserves Birkhoff-James orthogonality at each point of
        \[\overline{U}\cap \sm \mathbb{X}=\mathbb{X}\cap \sm \mathbb{X} = \sm \mathbb{X}.\]
         Therefore, it follows from Corollary \ref{sm k set} that $T$ is a scalar multiple of an isometry.
	\end{proof}
Due to the homogeneity of Birkhoff–James orthogonality, the following corollary follows directly from the above theorem.
    \begin{cor}\label{H preseves}
		Let $\sm\mathbb{X}$ be a  dense $G_{\delta}$ subset of $\mathbb{X}.$  Let $H\subset \mathbb{X}$ be a hyperplane that does not pass through the origin. If $T\in \mathbb{L}(\mathbb{X})$ preserves Birkhoff-James orthogonality at each point of $H$ then $T$ is a scalar multiple of an isometry. 
	\end{cor}
	
	 Note that in a Banach space, if an operator preserves Birkhoff-James orthogonality on a hyperplane passing through the origin, it may not necessarily be a scalar multiple of an isometry. This is because such a hyperplane is a subspace of codimension $1,$ ensuring the existence of a non-zero operator whose kernel contains that hyperplane.
	
	\begin{remark}\label{equivalence}
		For a Banach space $\mathbb{X},$ if $x\in \sm \mathbb{X}$ then it is easy to observe that for any non-zero operator $T\in\mathbb{L}(\mathbb{X}),$ the following are true:	\begin{itemize}
			\item [(i)] If $T$ preserves $\rho_+$-orthogonality at $x$ then $T$ preserves Birkhoff-James orthogonality at $x.$
			\item [(ii)]If $T$ preserves $\rho_-$-orthogonality at $x$ then $T$ preserves Birkhoff-James orthogonality at $x.$
			\item [(iii)]If $T$ preserves $\rho$-orthogonality at $x$ then $T$ preserves  Birkhoff-James orthogonality at $x.$
		\end{itemize}
	\end{remark}
	By combining the results established so far, we arrive at the following characterization of isometries on  Banach spaces in which the set of  all smooth points forms a dense $G_{\delta}$ set.
	\begin{theorem}\label{norm-derivatives}
		Let $\mathbb{X}$ be a Banach space such that $\sm\mathbb{X}$ is a  dense $G_{\delta}$ subset of $\mathbb{X}.$ and $T\in\mathbb{L}(\mathbb{X})$ be non-zero. Then the following conditions are equivalent:
		\begin{itemize}
			\item [(i)] $T$ preserves Birkhoff-James orthogonality on $\sm \mathbb{X}.$
            \item[(ii)] $T$ preserves Birkhoff-James orthogonality at each point of  $U,$ where $U$ is a dense subset of $\mathbb{X}.$
			\item [(iii)] $T$ preserves Birkhoff-James orthogonality at each point of  $H,$ where $H$ is a hyperplane that does not pass through the origin.
			\item [(iv)] $T$ preserves $\rho_+$-orthogonality on $\sm \mathbb{X}.$
			\item [(v)] $T$ preserves $\rho_-$-orthogonality on $\sm \mathbb{X}.$
			\item [(vi)] $T$ preserves $\rho$-orthogonality on $\sm \mathbb{X}.$
			\item [(vii)] $T$ is a scalar multiple of an isometry. 
		\end{itemize}
	\end{theorem}

	Now we turn our attention to the preservation of Birkhoff-James orthogonality at extreme points of the unit ball of the Banach space. First, we have the following result.
	\begin{theorem}\label{span}
		Let $\mathbb{X}$ be a finite-dimensional Banach space and let $\mathbb{Y}$ be a proper subspace of $\mathbb{X}.$ Then $$ \spn \Big(\bigcap\limits_{x\in \ext B_{\mathbb{X}}\setminus \mathbb{Y}} J(x)\Big)=\mathbb{X}^*.$$ 
	\end{theorem}
	\begin{proof}
		First we show that \[\spn \Big(\bigcap\limits_{x\in \ext B_{\mathbb{X}}\setminus \mathbb{Y}} J(x)\Big)=\spn \Big(\bigcap\limits_{x\in  S_{\mathbb{X}}\setminus \mathbb{Y}} J(x)\Big).\]
		Clearly,
		$\spn \Big(\bigcap\limits_{x\in \ext B_{\mathbb{X}}\setminus \mathbb{Y}} J(x)\Big)\subset\spn \Big(\bigcap\limits_{x\in  S_{\mathbb{X}}\setminus \mathbb{Y}} J(x)\Big).$
		Let $f\in \spn \Big(\bigcap\limits_{x\in  S_{\mathbb{X}}\setminus \mathbb{Y}} J(x)\Big).$ Then there exists $w\in  S_{\mathbb{X}}\setminus \mathbb{Y}$ such that $f\in J(w).$
		Consider the face $F=\left\{x\in S_{\mathbb{X}}:f(x)=1\right\}.$  Since $F\not\subset\mathbb{Y},$ it follows that $\ext F\setminus \mathbb{Y}\neq \emptyset.$ Then there exists $z\in \ext F\setminus \mathbb{Y}\subset  \ext B_{\mathbb{X}}\setminus \mathbb{Y}$ such that $f\in J(z).$ Thus, $\spn \Big(\bigcap\limits_{x\in \ext B_{\mathbb{X}}\setminus \mathbb{Y}} J(x)\Big)\supset \spn \Big(\bigcap\limits_{x\in  S_{\mathbb{X}}\setminus \mathbb{Y}} J(x)\Big).$
		Consequently, 
		\[\spn \Big(\bigcap\limits_{x\in \ext B_{\mathbb{X}}\setminus \mathbb{Y}} J(x)\Big)=\spn \Big(\bigcap\limits_{x\in  S_{\mathbb{X}}\setminus \mathbb{Y}} J(x)\Big).\]
		
		If possible suppose that $\mathbb{V}=\spn \Big(\bigcap\limits_{x\in  S_{\mathbb{X}}\setminus \mathbb{Y}} J(x)\Big)$ is a proper subspace of $\mathbb{X}^*.$ Now, for each $f\in S_{\mathbb{X}^*}\setminus \mathbb{V},$ there exists $y\in S_{\mathbb{Y}}$ such that $f(y)=1.$\\
		Next, let $g\in  S_{\mathbb{X}^*}\cap \mathbb{V}.$ Since  $\mathbb{X}^*\setminus \mathbb{V}$ is dense in $\mathbb{X}^*,$ there exists a sequence $\{g_n\}\subset \mathbb{X}^*\setminus \mathbb{V}$ such that $g_n\longrightarrow g.$ Let for each $n\in \mathbb{N}, ~ y_n\in S_{\mathbb{Y}}$ such that $g_n(y_n)=\|g_n\|.$ Since  $S_{\mathbb{Y}}$ is compact, there exists a subsequence $\{y_{n_k}\}$ of $\{y_n\}$ such that $y_{n_k}\longrightarrow  y\in S_{\mathbb{Y}}.$  So $g_{n_k}(y_{n_k})\longrightarrow g(y)$ and $g_{n_k}(y_{n_k})= \|g_{n_k}\|\longrightarrow \|g\|=1.$ Hence $g(y)=1.$ Therefore, for each $f\in S_{\mathbb{X}^*}$ there exists  $x\in S_{\mathbb{Y}} $ such that $f(x)=1.$ This is a contradiction, as there always exists $h\in S_{\mathbb{X}^*}$ such that $h(y)=0$ for all $y\in \mathbb{Y}.$ This completes the proof the lemma.
	\end{proof}
	
	The above result leads us to the following corollary.
	
	\begin{cor}\label{bijective}
		Let $\mathbb{X}$ be a finite-dimensional Banach space. If any non-zero operator $T\in \mathbb{L}(\mathbb{X})$ preserves Birkhoff-James orthogonality at each $x\in \ext B_{\mathbb {X}}$ then $T$ is bijective. 
	\end{cor}
	\begin{proof}
		Let $\dim \mathbb{X}= n.$ If possible, suppose that $\ker T\neq \{0\}.$ From Theorem \ref{span},  it follows that there exist $x_1, x_2,\ldots,x_n\in \ext B_{\mathbb{X}} \setminus \ker T $ such that for each $1\leq i\leq n,$ there exists $f_i\in J(x_i),$ where $f_i$ are linearly independent. Now,
		$$\bigcap\limits_{1\leq i\leq n}\ker f_i= \{0\}.$$ 
		Then there exists  $1\leq j\leq n$ such that $\ker T\setminus \ker f_j\neq \emptyset.$ Let $z\in \ker T\setminus \ker f_j.$  Then $x_j= \alpha z+h$ for some $\alpha\in \mathbb{R}\setminus\{0\}$ and $h\in \ker f_j\setminus \{0\}$ and so $Tx_j=Th.$ Now,
        \begin{eqnarray*}
            x_j\perp_B h&\implies& Tx_j\perp_B Th\\
            &\implies& Tx_j\perp_B Tx_j\\
            &\implies& Tx_j=0
        \end{eqnarray*}
         which is a contradiction. Therefore, $T$ is bijective.
		
	\end{proof}

	\begin{remark}
		For an $n$-dimensional Banach space, if a non-zero operator $T\in \mathbb{L}(\mathbb{X})$ preserves Birkhoff-James orthogonality at $n$ linearly independent extreme points of $B_{\mathbb {X}}$ then $T$ may not be bijective. For example, consider the $3$-dimensional Hilbert space $\mathbb{H}.$ Then the operator $T\in \mathbb{L}(\mathbb{H})$ given by $T(x,y,z)=(x,y,0)$ preserves Birkhoff-James orthogonality at three linearly independent extreme points $(1,0,0),(0,1,0),(0,0,1)$ but $T$ is not bijective.
	\end{remark}

  Next, we have the following observation regarding the local preservation of Birkhoff-James orthogonality at any points on a face of the unit ball of a Banach space.
	\begin{lemma}\label{equi}
		  Let $T\in \mathbb{L}(\mathbb{X}, \mathbb{Y})$ preserve Birkhoff-James orthogonality at $u,v\in S_{\mathbb{X}}$. If $u,v$ are on the same face $F$ of $B_{\mathbb{X}}$ then $\|Tu\|=\|Tv\|.$
	\end{lemma}
	\begin{proof}
		Let $Tu=ky$ and $Tv=lz$ for some $y,z\in B_{\mathbb{X}}$ and for some $k,l\in \mathbb{R}.$ Let $f\in S_{\mathbb{X}^*}$ be a support functional corresponding to the face $F$ of $B_{\mathbb{X}}.$ Then $f(u-v)=\|u\|-\|v\|=0$ and so $u\perp_B(u-v)$ and $v\perp_B(u-v).$ Then $y\perp_B(ky-lz)$ and $z\perp_B(ky-lz).$ Clearly, $k=0\iff l=0.$ Let $k,l\neq 0.$  Now,
		\[y\perp_B(ky-lz)\implies \|y+\lambda(ky-lz)\|\geq\|y\|\text{ for all }\lambda \in \mathbb{R}.\] 
		Let $\lambda=-\frac{1}{k},$ then $\frac{|l| }{|k|}\geq1.$ Similarly, from $z\perp_{B}{(ky-lz)} ,$ we have  $\frac{|k|}{|l|}\geq1.$ Thus, $|l|=|k|$ and so $\|Tu\|=\|Tv\|.$
	\end{proof}

	Now we state a characterization of a subspace contained in the orthogonal region of a point in a Banach space, which will be needed for our next result.
	\begin{lemma}\cite[Lemma 2.14]{MMPS25}\label{subspace}
		Let  $x\in\mathbb{X}$ be non-zero and let $\mathbb{V}$  be a subspace of $\mathbb{X}.$ Then $\mathbb{V}\subset x^{\perp_B}$ if and only if there exists $f\in J(x)$ such that $\mathbb{V}\subset \ker f.$
	\end{lemma}

	We now present some consequences when any operator defined on a Banach space preserves Birkhoff-James orthogonality on a set lying within a face of the unit ball of the space.
	
	\begin{theorem}\label{faceToface}
		Let  $F$ be a face of $B_{\mathbb{X}}.$  If  $T\in \mathbb{L}(\mathbb{X},\mathbb{Y})$ preserves Birkhoff-James orthogonality on a set $A\subset F$ then the following results hold:
		\begin{itemize}
			\item[(i)]  Either $T(A)=0$ or $\frac{1}{k}T (A)\subset G,$  where $G$  is a  face of $B_{\mathbb{Y}}$ and for some $k\in \mathbb{R^+}.$ 
			\item[(ii)] $T$ preserves Birkhoff-James orthogonality on $\co(A).$
		\end{itemize}
	\end{theorem}
	\begin{proof}
		(i) Let $T\in \mathbb{L}(\mathbb{X},\mathbb{Y})$ preserve Birkhoff-James orthogonality on a set $A\subset F.$ Let $f$ be a support functional corresponding to $F.$ Suppose that $A=\{u_{\alpha}:\alpha\in\Lambda\},$ where $\Lambda $ is the index set. Since  $T$ preserves Birkhoff-James orthogonality on  $A,$ it follows that $T(\ker f)\subset(Tu_{\alpha})^{\perp_{B}}$ for each $\alpha\in \Lambda.$ From Lemma \ref{equi}, it follows that $\|Tu_{\alpha}\|=\|Tu_{\beta}\|=k$(say) for any $\alpha,\beta \in \Lambda.$ If $k=0$ then $T(A)=0.$ Let $k\neq0.$ Let $\beta\in \Lambda$ then  from \ref{subspace}, it follows that  there exists $g\in J(Tu_{\beta})$  such that $T(\ker f)\subset\ker g.$  We claim  that  $g\in J(Tu_{\alpha})$ for all $\alpha\in \Lambda.$ Clearly, $\mathbb{X}=\ker f \oplus \spn\{u_{\beta}\}.$ Hence for each $\alpha \in \Lambda $ there exists $h\in \ker f$ and $\lambda_{\alpha}\in \mathbb{R}$ such that $u_{\alpha}=h_{\alpha}+\lambda_{\alpha} u_{\beta}.$ Then 
		\[f(u_{\alpha})=f(h+\lambda_{\alpha} u_{\beta})
		\implies \lambda_{\alpha}=1~\text{ for all  } \alpha \in \Lambda.\]
		Thus,  for each $\alpha \in \Lambda $ there exists $h_{\alpha}\in \ker f$  such that $u_{\alpha}=h_{\alpha}+u_{\beta}.$ Then for each $\alpha\in\Lambda,$ 
		\[g(Tu_{\alpha})=g(T(h_{\alpha}+u_{\beta}))=g(Tu_{\beta})=\|Tu_{\beta}\|=\|Tu_{\alpha}\|.\]
		Thus, our claim is established that $g\in J(Tu_{\alpha})$ for each $\alpha\in \Lambda.$ Let $G$ be the face of $B_{\mathbb{Y}}$ supported by $g.$ Therefore, $\frac{T}{k} (A)\subset G.$ \\
		
		(ii)  Let  $u\in \co(A).$ Then $u=\sum\limits_{1\leq i\leq n}a_iu_i,$ where $u_i\in A$ and $a_i\in \mathbb{R}^+$ with $\sum\limits_{1\leq i\leq n}a_i=1.$ Let $u\perp_{B}z,$ then there exists  $f\in J(u)$ such that $f(z)=0.$ Now 
		\begin{eqnarray*}
			f(u)&=&\sum\limits_{1\leq i\leq n}a_if(u_i)\\
			\implies 1&=&\sum\limits_{1\leq i\leq n}a_if(u_i).
		\end{eqnarray*} This shows that $f(u_i)=1$ for all $1\leq i\leq n.$ So $f\in J(u_i)$ for all $1\leq i\leq n.$ 
		Since $T$ preserves Birkhoff-James orthogonality at $u_i$ and $u_i\in F$ for all $1\leq i\leq n,$ it follows from Lemma \ref{equi} that $\|Tu_i\|=\|Tu_j\|=k$(say) for all $1\leq i, j\leq n.$  If $k=0$ then it is easy to observe that $Tu=0$ and so $T$ preserves Birkhoff-James orthogonality at $u.$ Suppose that $k\neq 0.$
		From Lemma \ref{subspace}, it follows that $T(\ker f)\subset (Tu_1)^{\perp_{B}}.$ This shows that there exists $g\in J(Tu_1)$ such that $T(\ker f)\subset \ker g.$   Then by similar argument used in (i), we have $g\in J(Tu_i)$ for all $1\leq i\leq n.$ Now
		\begin{eqnarray*}
			\|Tu\|\geq g(Tu)&=& g\big(\sum\limits_{1\leq i\leq n}a_iTu_i\big)\\
			&=&\sum\limits_{1\leq i\leq n}a_ig(Tu_i)\\
			&=&\sum\limits_{1\leq i\leq n}a_i\|Tu_i\|\\
			&\geq& \|\sum\limits_{1\leq i\leq n}a_iTu_i\|
			=\|Tu\|.
		\end{eqnarray*}
		This shows that $g\in J(Tu).$ Also we have $Tz\in T(\ker f)\subset \ker g.$ Therefore, $Tu\perp_B Tz$ and this completes the proof.
		
	\end{proof}
	The extreme points of the unit ball of a Banach space play an important role in characterizing the isometries on a Banach space. The following theorem  illustrates their significance by offering a substantial refinement of the Blanco-Koldobsky-Turn\v{s}ek characterization of isometries.
	
	\begin{theorem}\label{extreme}
		Let $\mathbb{X}$ be a Banach space such that $B_{\mathbb{X}}=\overline{\co(\ext B_{\mathbb{X}})}$ and $\sm\mathbb{X}$ is a  dense $G_{\delta}$ subset of $\mathbb{X}.$ If $T\in \mathbb{L}(\mathbb{X})$ preserves Birkhoff-James orthogonality on $\ext B_{\mathbb{X}}$ then $T$ is a scalar multiple of an isometry, i.e., $\ext B_{\mathbb{X}}$ is a $\mathcal{K}$-set.
	\end{theorem}
	\begin{proof}
		Let $T\in \mathbb{L}(\mathbb{X})$ preserve Birkhoff-James orthogonality at each  point of $\ext 
        B_{\mathbb{X}}.$  Let $z\in  \co(\ext B_{\mathbb{X}})\cap S_{\mathbb{X}}.$  Then there exists $m\in \mathbb{N}$ such that
        \[z=\sum\limits_{1\leq i\leq m}a_iz_i,\text{ where }
        z_i\in \ext B_{\mathbb{X}}\text{ and }a_i\in \mathbb{R}^+\text{ with }\sum\limits_{1\leq i\leq n}a_i=1.\]  Consider a face $F$ of $B_{\mathbb{X}}$ containing $z.$ Clearly, $z_i\in F$ for all $1\leq i\leq n.$ Then it follows from Theorem \ref{faceToface} (ii) that $T$ preserves Birkhoff-James orthogonality  at $z.$ By the homogeneity of  Birkhoff-James orthogonality, we can say that $T$ preserves Birkhoff-James orthogonality  at each point of $\co(\ext B_{\mathbb{X}}).$   Now it follows from Theorem \ref{sm pre} that $T$ preserves Birkhoff-James orthogonality at each point of \[\overline{\co(\ext B_{\mathbb{X}})}\cap \sm \mathbb{X}=B_{\mathbb{X}}\cap  \sm \mathbb{X}.\] 
        Again from the homogeneity of Birkhoff James orthogonality, it follows that $T$ preserves  Birkhoff-James orthogonality  at each point of $\sm \mathbb{X}.$ Therefore, it follows from Corollary \ref{sm k set} that $T$ is an isometry multiplied by a constant.
	\end{proof}
    The following corollary is an immediate consequence of the above theorem.
    \begin{cor}\label{finite extreme}
		Let $\mathbb{X}$ be a finite-dimensional Banach space. Then $\ext B_{\mathbb{X}}$ is a $\mathcal{K}$-set.
	\end{cor}

	\begin{remark}
    \noindent
    \begin{itemize}
        \item[(i)]  For an infinite-dimensional  Banach space $\mathbb{X}$, the set $\ext B_{\mathbb{X}}$ may not be a $\mathcal{K}$-set. For example, suppose $\mathbb{X}= \ell_1\oplus_1c_0,$ where for any $\widetilde{x}=(x_1, x_2) \in \mathbb{X},$ $\|\widetilde{x}\|_{\mathbb{X}}= \|x_1\|_{\ell_1} + \|x_2\|_{c_0}$. From \cite[Th. 5]{DS08}, it follows that $\ext B_{\mathbb{X}}=\{(u,0): u\in \ext B_{\ell_1}\}.$ Consider the operator $T\in \mathbb{L}(\mathbb{X})$ defined by $T(x,y)=(0,y).$ Clearly, $\ext B_{\mathbb{X}}\in \ker T$ and so $T$ preserves Birkhoff-James orthogonality at each point of $\ext B_{\mathbb{X}},$ but $T$ is not a scalar multiple of an isometry. Therefore, $\ext B_{\mathbb{X}}$ is not a $\mathcal{K}$-set.
        \item[(ii)] For  a  Banach space  $\mathbb{X},$  if an operator $T\in \mathbb{L}(\mathbb{X})$ preserves $\rho$-orthogonality at each extreme point of $B_{\mathbb{X}}$ then $T$ may not be a scalar multiple of an isometry. For example, consider $\mathbb{X}=\ell_{\infty}^2,$ then  the operator $T\in \mathbb{L}(\mathbb{X})$ defined by $T(x,y)=(x+y,y-x)$ preserves $\rho$-orthogonality at each extreme point of $B_{\mathbb X}$ but $T$ is not a scalar multiple of an isometry. Also, from the example it is clear that the image of an extreme point of $B_{\mathbb{X}}$ may not be an extreme point of $B_{\mathbb{X}}$ under the preservation of $\rho$-orthogonality  at each extreme point of $B_{\mathbb X}.$ 
        
    \end{itemize}
		    
	\end{remark}
	 In the next theorem, we provide a necessary condition for the preservation of Birkhoff-James orthogonality by a bijective operator between polyhedral Banach spaces on the relative interior of any facet of the unit ball of the domain space.
	
	\begin{theorem}\label{int facet}
		Let $\mathbb{X},\mathbb{Y}$ be $n$-dimensional polyhedral Banach spaces and let $F$ be a facet of $B_{\mathbb{X}}.$ Let a bijective operator $T\in \mathbb{L}(\mathbb{X},\mathbb{Y})$ preserve Birkhoff-James orthogonality at each point of $\rint F.$ Let $u\in \rint F.$ Then $\frac{T}{\|Tu\|}(\rint F)\subset \rint G$ for some facet $G$ of $B_{\mathbb{Y}}.$
	\end{theorem}
	
	\begin{proof}
		Let $\|Tu\|=k.$ From  Lemma \ref{equi} and Theorem \ref{faceToface} (i), it follows that $\|Tx\|=k$ for all $x\in \rint F$ and $\frac{Tx}{k}\in G$ for some face $G$ of $B_{\mathbb{X}}.$
		Thus, $\frac{T}{k}(F)\subset G.$  Since $F$ is a facet of $B_{\mathbb{X}},$ it follows that $F$ has $n$ linearly independent elements $x_1,x_2,\dots, x_n$ of $\mathbb{X}.$  For each $1\leq i\leq n,$ suppose that $Tx_i=y_i.$  Since $T$ is bijective, $y_1,y_2,\dots, y_n$ are  linearly independent elements of $G.$ Therefore, $G$ is a facet of $B_{\mathbb{Y}}.$ Let $T_1=\frac{T}{k}.$  Now we show that $T_1(\rint F)\subset \rint G.$ Let $v\in \rint F,$ then there exists an open set $D$ such that $v\in D\cap \aff(F)\subset F.$ Then 
        \[T_1v\in T_1(D\cap \aff(F))=T_1(D)\cap T_1(\aff(F)).\] 
        We claim that $T_1(\aff(F))=\aff(G).$ Since $T_1(F)\subset G,$  it follows that $T_1(\aff(F)) \subset \aff(G).$  Let $y\in \aff(G)=\aff(\{y_1,y_2,\dots, y_n\}),$ then there exist $a_i\in \mathbb{R}$ with $\sum\limits_{1\leq i\leq n}a_i=1$ such that
		\[ y=\sum\limits_{1\leq i\leq n}a_iy_i=\sum\limits_{1\leq i\leq n}a_iTx_i=T_1(\sum\limits_{1\leq i\leq n}a_ix_i)\subset T_1(\aff(F)).\] 
		Thus,  $T_1(\aff(F))=\aff(G).$    Next,
        \begin{eqnarray*}
          Tv\in T_1(D\cap \aff(F))&=&T_1(D)\cap T_1(\aff(F))\\
          &=&T_1(D)\cap \aff(G)\\
          &\subset& G.  
        \end{eqnarray*}
		Since $T_1$ is bijective, by the open mapping theorem we have $T_1(D)$ is an open set of $\mathbb{Y}.$ Thus, $T_1v\in \rint G.$ Therefore,  $T_1(\rint F)\subset \rint G$ and this completes the proof.
		
	\end{proof}
	
	The following result follows immediately from the above theorem.
	
	\begin{cor}
		Let $\mathbb{X},\mathbb{Y}$ be $n$-dimensional polyhedral Banach spaces and $F$ be a facet of $B_{\mathbb{X}}.$ Let a bijective operator $T\in \mathbb{L}(\mathbb{X},\mathbb{Y})$ preserve Birkhoff-James orthogonality at each point of $F\cap \sm\mathbb{X}.$ Let $u\in F\cap \sm\mathbb{X}.$ Then $\frac{T}{\|Tu\|}(F\cap \sm\mathbb{X})=G\cap \sm\mathbb{Y}$ for some facet $G$ of $B_{\mathbb{Y}}.$
	\end{cor}
	
	\begin{remark}
		We note that in Theorem \ref{int facet}, if the operator $T$ fails to be bijective then the theorem may not hold. For example, let the Banach space $\mathbb{X}=\ell_{\infty}^{2}$ and let the operator $T\in \mathbb{L}(\mathbb{X})$ defined by $T(x,y)=(x,x)$ for all $(x,y)\in\mathbb{X}.$ Clearly, $T$ is not bijective. Now consider the facet $F=\{(1,a): |a|\leq 1\}.$ It is easy to observe that $T$ preserves Birkhoff-James orthogonality at each point of $\rint F$ but there does not exist any facet $G$ of the unit ball of $\mathbb{X}$ such that $T(\rint F)\subset \rint G.$ 
	\end{remark}
	
	Our next goal is to show that for any fixed $1\leq k\leq n$ the set of all $k$-smooth points of the unit ball of an $n$-dimensional polyhedral Banach space is also a $\mathcal{K}$-set. For this purpose, we have two nice and easy observations corresponding to a facet of the unit ball of the concerned space.
	
	\begin{proposition}\label{Prop sub face}
		Let $\mathbb{X}$ be an $n$-dimensional polyhedral Banach space and let $1\leq k\leq n.$ Let $F$ be a $k$-face of $B_{\mathbb{X}}$ and let $x\in \ext F.$ Then there exists a $(k-1)$-face $G$ of $B_{\mathbb{X}}$ such that $x\in G \subset F.$
	\end{proposition}
	\begin{proof}
		Let $F$ be a $k$-face of $B_{\mathbb{X}}.$  Consider the subspace $\mathbb{V}=\spn\{u-v:u,v\in F\}.$ Clearly, $\dm\mathbb{V}=k.$ Then $F'=\spn\{u-x:u\in F\}$ is a convex polyhedron contained in $\mathbb{V}$ of dimension $k.$  Since $x$ is an extreme point of $F,$ it follows that $0$ is an extreme point of $F'.$  Let $G'$ be a $(k-1)$-face of $F'$ containing $0.$ Now let us consider the set $G=\spn\{w+x:w\in G'\}.$ Then clearly $G\subset F$ is a convex polyhedron of dimension $k-1$ containing $x.$
		Now we show that $G$ is a face of $F.$ Let $y\in G $ be such that $y=(1-t)y_1+ty_2$ for  $y_1,y_2\in F$ and for  $t\in (0,1).$ Now $y\in G\implies y-x\in G'$ and $y_1,y_2\in F\implies y_1-x,y_2-x\in F'.$ This shows that 
		\[y-x=(1-t)(y_1-x)+t(y_2-x).\]
		Since $G'$ is a face of $F',$ it follows that $y_1-x,y_2-x\in G'$ and so $y_1,y_2\in G.$ Thus, $G$ is a face of $F$ of dimension $k-1$ containing $x$ and consequently, $G$ is a $(k-1)$-face of $B_{\mathbb{X}}$ such that $x\in G \subset F.$  
	\end{proof}
	
	\begin{proposition}\label{k-sm closer}
		Let $\mathbb{X}$ be an $n$-dimensional polyhedral Banach space. Let $F$ be a facet of $B_{\mathbb{X}}.$ Then $\ext F\subset \overline{\ksm \mathbb{X}\cap F}$ for any $k<n.$ 
	\end{proposition}
	\begin{proof}
		Let $x\in \ext F.$ From Proposition \ref{Prop sub face}, it follows that there exists a $(n-k)$-face $G$ of $B_{\mathbb{X}}$ such that $x\in G\subset F.$ Let $y\in \rint G$ and so $ y\in\ksm\mathbb{X}\cap F$ such that $J(y)\subset J(x).$ Consider the sequence $\{x_n\},$ where $x_n=(1-\frac{1}{n})x+\frac{1}{n}y$ for all $n\in\mathbb{N}.$ Clearly, $x_n\in F$ for all $n\in\mathbb{N}$ and $x_n\longrightarrow x,$ as $n\longrightarrow \infty.$ Now we show that $J(x_n)=J(y)$ for all $n\in\mathbb{N}.$ Let $f\in J(y),$ then $f\in J(x).$ Hence 
        \begin{eqnarray*}
            f(x_n)&=&f\bigg(\Big(1-\frac{1}{n}\Big)x+\frac{1}{n}y\bigg)\\
            &=&\Big(1-\frac{1}{n}\Big)+\frac{1}{n}=1
        \end{eqnarray*}
		Thus, $f\in J(x_n)$ and so $J(y)\subset J(x_n)$  for all $n\in\mathbb{N}.$ Next, let $g\in J(x_n),$ then 
        \begin{eqnarray*}
            1=g(x_n)&=&\Big(1-\frac{1}{n}\Big)g(x)+\frac{1}{n}g(y)\\
            &\leq& \Big(1-\frac{1}{n}\Big)+\frac{1}{n}=1.
        \end{eqnarray*}
		This implies that $g(y)=1=\|y\|$ and so $g\in J(y).$ Therefore, $J(x_n)\subset J(y)$ for all $n\in \mathbb{N}.$ Thus, $J(x_n)=J(y)$ for all $n\in\mathbb{N}.$ Therefore, $\{x_n\}\subset \ksm\mathbb{X}\cap F$ and hence $x\in \overline{\ksm\mathbb{X}\cap F}.$
	\end{proof}

    From Corollary \ref{finite extreme}, it follows that any operator that preserves Birkhoff-James orthogonality on the set of all $n$-smooth points of an $n$-dimensional polyhedral Banach space must be an isometry. We are now prepared to extend this result by showing that for any $1\leq k\leq n,$  Birkhoff-James orthogonality preservation on the set of all $k$-smooth points similarly ensures the operator is an isometry.
	\begin{theorem}\label{k-SM}
		Let $\mathbb{X}$ be a finite-dimensional polyhedral Banach space and let $1\leq k\leq \dim \mathbb{X}.$ If $T\in \mathbb{L}(\mathbb{X})$ preserves Birkhoff-James orthogonality at each $k$-smooth point of $B_{\mathbb{X}}$ then $T$ is a scalar multiple of an isometry, i.e., $\ksm\mathbb{X}$ is a $\mathcal{K}$-set.
	\end{theorem}
	\begin{proof}
		If $T=0,$ the theorem follows trivially.
		Let $T(\neq 0)\in \mathbb{L}(\mathbb{X})$ preserve Birkhoff-James orthogonality on $\ksm\mathbb{X}.$ Let $F$ be a facet of $B_{\mathbb{X}}.$ From Lemma \ref{equi}, it follows that $\|Tu_1\|=\|Tu_2\|=l_{F}$(say) for any $u_1,u_2\in \ksm\mathbb{X}\cap F$ and so $\|Tu\|=l_{F}$ for all $u\in \overline{\ksm\mathbb{X}\cap F}.$ Then by Proposition \ref{k-sm closer}, we
		have $\|Tu\|=l_{F}$ for any $u\in \ext F.$
		Next, it follows from  Theorem \ref{faceToface} (i)  that  $\frac{1}{l_{F}}T(\ksm\mathbb{X}\cap F)\subset G$ for some face $G$ of $B_{\mathbb{X}}.$ Thus,
        \begin{eqnarray*}
            \frac{1}{l_{F}}T(F)&=&\frac{1}{l_{F}}T(\co(\ext F))\\
            &=&\co(\frac{1}{l_{F}}T(\ext F))\\
            &\subset& \co(\frac{1}{l_{F}}T(\overline{\ksm\mathbb{X}\cap F}))\\
            &\subset& G.
        \end{eqnarray*}
		Therefore, $\|Tu\|=l_{F}$ for all $u\in F.$  Since $F$ is arbitrary, we can say that this is true for any  facet  of $B_{\mathbb{X}}.$
 Note that each extreme point of $B_{\mathbb{X}}$ belongs to the intersection of at least $n(=\dm\mathbb{X})$ number of distinct  facets and it is easy to observe that for any facet $F$ of $B_{\mathbb{X}},$ $\ext F\subset \ext B_{\mathbb{X}}.$  Thus,  we can conclude that $\|Tu\|$ is constant for all $u\in S_{\mathbb{X}}$ and this completes the proof.
	\end{proof}
	\begin{remark}
		
		In general, for an $n$-dimensional Banach space, the above theorem may not hold because $k$-smooth points may not exist for all $1 < k \leq n$. Even if such points exist for some $1 < k \leq n$, the set $\ksm\mathbb{X}$ may still fail to be a $\mathcal{K}$-set. For example, consider the 3-dimensional Banach space $\mathbb{X},$ whose unit sphere is described by the following figure.
		
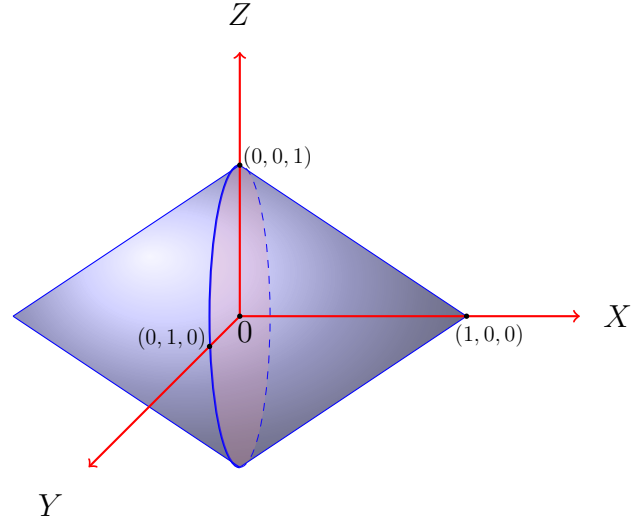
\begin{center}
			\begin{tikzpicture}
				\shade[ball color = blue!30, opacity = 0.8] (0,2) -- (3,0) -- (0,-2)--(-3,0)--(0,2);
				\draw [rotate=90,dashed,blue] (-2,0) arc (180:360:2 and 0.4);
				\draw[rotate=90,blue,thick] (2,0) arc (0:180:2 and 0.4);
				\fill[rotate=90,red!40,opacity=0.2] (0,0)  ellipse (2 and 0.4);
				\draw[blue](0,2)--(3,0);
				\draw[blue](0,-2)--(3,0);
				\draw[blue](0,2)--(-3,0);
				\draw[blue](0,-2)--(-3,0);
				\draw[->, red,thick ](0,0)--(0,3.5);
				\draw[->, red,thick ](0,0)--(4.5,0);
				\draw[->, red,thick ](0,0)--(-2,-2);
				\node [black] at (0,4) {$Z$};	
				\node [black] at (5, 0) {$X$};	
				\node [black] at (-2.5, -2.5) {$Y$};	
				
				\fill[black] (3,0) circle(1pt);
				\node [black,scale=.7] at (3.3,-0.25) {$(1,0,0)$};
				\fill[black] (0,0) circle(1pt);
				\node [black] at (0.07,-0.2) {$0$};
				\fill[black] (0,2) circle(1pt);
				\node [black,scale=.7] at (0.5,2.1) {$(0,0,1)$};
				\fill[black] (-.4,-.4) circle(1pt);
				\node [black,scale=.7] at (-.9,-.3) {$(0,1,0)$};
				
			\end{tikzpicture}
		\end{center}
		This unit sphere is a surface of revolution obtained by rotating a square with vertices $\{\pm(1,0,0), \pm(0,0,1)\}$ around the $X$-axis.
		Here $2\text{-Sm~}\mathbb{X}=\{(0,b,c): b,c\in\mathbb{R},(0,b,c)\neq(0,0,0) \}$  and $3\text{-Sm~}\mathbb{X}=\{(a,0,0): a\in\mathbb{R}\setminus\{ 0\}\}.$ Since $\dim\spn(2\text{-Sm~}\mathbb{X})=2$ and $\dim\spn(3\text{-Sm~}\mathbb{X})=1,$ it follows from \cite[Prop. 2.1]{SMP24} that $2\text{-Sm~}\mathbb{X}$ and $3\text{-Sm~}\mathbb{X}$ are not $\mathcal{K}$-sets.
	\end{remark}

	We end this article by providing the following characterization of isometries on a finite-dimensional polyhedral Banach space, which easily follows from  Theorem \ref{norm-derivatives} and Theorem \ref{k-SM}.
	\begin{theorem}
		Let $\mathbb{X}$ be a finite-dimensional polyhedral Banach space. Let $T\in \mathbb{L}(\mathbb{X}),$ then the following results are equivalent:
		\begin{itemize}
			
			\item[(i)]$T$ preserves Birkhoff-James orthogonality at each $k$-smooth point of $S_{\mathbb{X}},$ where $1\leq k\leq \dim \mathbb{X}.$
            \item[(ii)] $T$ preserves Birkhoff-James orthogonality at each point of  $U,$ where $U$ is a dense subset of $S_{\mathbb{X}}.$
			\item[(iii)] $T$ preserves $\rho$-orthogonality at each smooth point of $S_{\mathbb{X}}.$
			\item[(iv)] $T$ preserves $\rho_+$-orthogonality at each smooth point of $S_{\mathbb{X}}.$
			\item[(v)] $T$ preserves $\rho_-$-orthogonality at each smooth point of $S_{\mathbb{X}}.$
			\item[(vi)] $T$ is a scalar multiple of an isometry.
			
		\end{itemize} 
	\end{theorem}
		\noindent 	\textbf{Conflict of interest}\\
		The authors have no relevant financial or non-financial interests to disclose.
		The authors have no competing interests to declare that are relevant to the content of this article.

\end{document}